\documentclass[12pt]{amsart}
\usepackage{amsfonts,latexsym,amscd}

\usepackage[all,cmtip,2cell]{xy}
\UseAllTwocells
\usepackage{verbatim}
\usepackage{hyperref}
\usepackage{mathrsfs}
\usepackage{amssymb}
\usepackage{stmaryrd}

\numberwithin{equation}{section}

\newtheorem{theorem}{Theorem}[section]
\newtheorem{lemma}[theorem]{Lemma}
\newtheorem{proposition}[theorem]{Proposition}
\newtheorem{corollary}[theorem]{Corollary}
\theoremstyle{definition}

\newtheorem{example}[theorem]{Example}

\newtheorem{remark}[theorem]{Remark}

\newcommand{\id}{{\rm id}}

\newcommand{\Mod}{{\rm Mod}}
\newcommand{\Fun}{{\rm Fun}}
\newcommand{\Irr}{\text{\rm Irr}}

\newcommand{\FPdim}{{\rm FPdim}}
\newcommand{\Ext}{{\rm Ext}}

\newcommand{\Hom}{{\rm Hom}}

\newcommand{\Rep}{{\rm Rep}}

\newcommand{\sRep}{{\rm sRep}}
\newcommand{\sVec}{{\rm sVec}}

\newcommand{\Vect}{{\rm Vec}}

\newcommand{\B}{\mathcal{B}}

\newcommand{\la}{\langle\,}
\newcommand{\ra}{\,\rangle}

\newcommand{\ot}{\otimes}

\newcommand{\ben}{\begin{enumerate}}
\newcommand{\een}{\end{enumerate}}

\newcommand{\CC}{{\mathbb{C}}}

\newcommand{\C}{{\mathcal C}}

\newcommand{\E}{\mathcal{E}}

\newcommand{\Z}{{\mathcal Z}}

\newcommand{\A}{{\mathcal A}}
\newcommand{\N}{{\mathcal N}}

\newcommand{\D}{{\mathcal D}}

\newcommand{\be}{{\bf 1}}

\newcommand{\Bimod}{{\rm Bimod}}

\begin{document}

\title[Braided categories with a Lagrangian subcategory]{On finite non-degenerate braided tensor categories with a Lagrangian subcategory} 

\author{Shlomo Gelaki}\thanks{Current affiliation of Shlomo Gelaki: Department of Mathematics, Iowa State University, Ames, IA 50011, USA. Email: gelaki@iastate.edu}

\address{Department of Mathematics, Technion-Israel Institute of
Technology, Haifa 32000, Israel} \email{gelaki@math.technion.ac.il}

\author{Daniel Sebbag}
\address{Department of Mathematics, Technion-Israel Institute of
Technology, Haifa 32000, Israel}
\email{sebbag.daniel@gmail.com}
\date{\today}


\begin{abstract}
Let $W$ be a finite dimensional vector space over $\mathbb{C}$ viewed as a purely odd supervector space, and let $\sRep(W)$ be the finite symmetric tensor category of finite dimensional superrepresentations of the finite supergroup $W$. We show that the set of equivalence classes of finite non-degenerate braided tensor categories $\C$ containing $\sRep(W)$ as a Lagrangian subcategory is a torsor over the cyclic group $\mathbb{Z}/16\mathbb{Z}$. In particular, we obtain that there are $8$ non-equivalent such braided tensor categories $\C$ which are integral and $8$ which are non-integral.
\end{abstract}

\maketitle

\tableofcontents

\section{Introduction}
There are two types of finite braided tensor categories,  which in a sense are opposite to each other. On one extreme  we have the symmetric tensor categories, i.e., braided tensor categories which coincide with their M\"{u}ger centralizer, while on the other extreme we have the \emph{non-degenerate} braided tensor categories, i.e., braided tensor categories with trivial M\"{u}ger centralizer.

On one hand, finite symmetric tensor categories over $\CC$ are completely classified in supergroup-theoretical terms. Namely, by a theorem of Deligne \cite{De}, every such category is braided tensor equivalent to the category $\sRep(G\ltimes{W},u)$ of finite dimensional superrepresentations of a unique (up to isomorphism) finite supergroup $G\ltimes W$ with a central element $u\in G$ of order $\le 2$ acting via the parity automorphism (see  \cite{EG1,EGNO}). 

On the other hand, the classification of non-degenerate braided {\em fusion} categories $\C$ which contain a Tannakian subcategory $\Rep(G)$ as a {\em Lagrangian} subcategory (i.e., $\Rep(G)$ coincides with its M\"{u}ger centralizer inside $\C$) is known too. Namely, such categories $\C$ are precisely the centers $\mathcal{Z}(\Vect^{\omega}(G))$ of pointed fusion categories $\Vect^{\omega}(G)$ \cite{DGNO, DGNO2}.

Furthermore, the classification of non-degenerate braided {\em fusion} categories $\B$ which contain the symmetric tensor category $\sVec$ of finite dimensional supervector spaces as a {\em Lagrangian} subcategory is also known \cite{DGNO, DGNO2}. Namely, it is known that there are exactly $16$ such categories $\B$, up to {\em braided tensor} equivalence. Moreover, the categories $\B$ form a group ${\bf B}$ isomorphic to $\mathbb{Z}/16\mathbb{Z}$ with respect to a certain modified Deligne tensor product $\widetilde{\boxtimes}$ \cite{DMNO,DNO} (see \ref{MissVec} for more details).

Our purpose in this paper is to take the first step in an attempt to extend the classification of \cite{DGNO, DGNO2} to finite (non-semisimple) non-degenerate braided tensor categories containing a Lagrangian subcategory. Namely, we classify finite non-degenerate braided tensor categories $\C$ containing $\sRep(W):=\sRep(\langle u\rangle \ltimes{W},u)$ as a Lagrangian subcategory for a finite dimensional vector space $W$ viewed as a {\em purely odd} supervector space. Observe that the center $\Z(\sRep(W))$ of $\sRep(W)$ is an example of such category $\C$ (see Theorem \ref{thecenterisnondeg}).

More precisely, first we prove in Theorem \ref{pre main theorem} that every finite non-degenerate braided tensor category containing a Lagrangian subcategory $\sRep(W)$ admits a natural $\mathbb{Z}/2\mathbb{Z}$-faithful grading, and has $2$ invertible objects and exactly $1$ or $2$ more simple objects (noninvertible, if $W\ne 0$). 

We then show in Theorem \ref{freeaction} that the group ${\bf B}\cong \mathbb{Z}/16\mathbb{Z}$ acts freely on the set of equivalence classes of finite non-degenerate braided tensor categories $\C$ containing $\sRep(W)$ as a Lagrangian subcategory. 

Finally, we use Theorems \ref{pre main theorem}, \ref{freeaction} to prove the following theorem, which is the main result of this paper.

\begin{theorem}
\label{main theorem}
The following hold:
\begin{enumerate}
\item The action of the group ${\bf B}\cong \mathbb{Z}/16\mathbb{Z}$ on the set of equivalence classes of finite non-degenerate braided tensor categories containing $\sRep(W)$ as a Lagrangian subcategory is free and transitive.

\item There are $8$ equivalence classes of finite non-degenerate braided {\em integral} tensor categories containing $\sRep(W)$ as a Lagrangian subcategory, and $8$ equivalence classes of finite non-degenerate braided {\em non-integral} tensor categories containing $\sRep(W)$ as a Lagrangian subcategory.
\end{enumerate}
\end{theorem}

In particular, Theorem \ref{main theorem} yields precise information on the number and projectivity of the simple objects in a finite non-degenerate braided tensor category containing a Lagrangian subcategory $\sRep(W)$ (see Corollary \ref{simprojex}). 

The structure of this paper is as follows. Section \ref{section:Preliminaries} is devoted to some preliminaries on finite (braided) tensor categories and their exact module categories, Hopf superalgebras, the finite tensor categories $\sRep(W)$, and the non-degenerate braided tensor categories $\Z(\sRep(W))$. Section \ref{section:proof of pre main theorem (1)} is devoted to the proof of Theorem \ref{pre main theorem}. Section \ref{MissVec} is devoted to the group ${\bf B}$, and to the proof that it acts on the set of equivalence classes of finite non-degenerate braided tensor categories containing the same $\sRep(W)$ as a Lagrangian subcategory, freely. Section \ref{section:proof of case (3)} is devoted to the proof of Theorem \ref{main theorem}. In Section \ref{section:proof of case (2)} we classify finite {\em degenerate} braided tensor categories containing a Lagrangian subcategory $\sRep(W)$. In Section \ref{appendix} we relate Theorem \ref{main theorem} with the works of Davydov and Runkel \cite{DR1,DR,DR2}.

\begin{remark}
In a future publication, we plan to extend the results of this paper to finite non-degenerate braided tensor categories containing a Lagrangian subcategory $\sRep(G\ltimes{W},u)$ of the most general form.
\end{remark}

\medskip
{\bf Acknowledgements.} We are grateful to Alexei Davydov, Pavel Etingof,     
Dmitri Nikshych, Victor Ostrik and Ingo Runkel for very useful discussions and helpful comments.

This work was done under the supervision of the first author, as part of the second author's Ph.D dissertation. 

Part of this work was done while the first author was visiting the Department of Mathematics at the University of Michigan in Ann Arbor; he is grateful for their warm hospitality.

This work was partially supported by the Israel Science Foundation (grant no. 561/12).

\section{Preliminaries}\label{section:Preliminaries}

Throughout this paper, the ground field will be the field 
$\CC$ of complex numbers, and all categories will be assumed to be $\CC$-linear abelian.
We refer the reader to the book \cite{EGNO} for a general background on finite tensor categories.

\subsection{Finite tensor categories}\label{finitetensorcat}

Let $\C$ be a finite tensor category, and let ${\rm Gr}(\C)$ be the Grothendieck ring of $\C$. 
Recall \cite[Subsection 2.4]{EO} that we have a character $\FPdim: {\rm Gr}(\C)\to \mathbb{R}$, attaching to $X\in \C$ the {\em Frobenius-Perron dimension} of $X$. 
Following \cite[Subsection 2.4]{EO}, we set $\FPdim(\C):=\sum_{X\in\Irr(\C)} \FPdim(X)\FPdim(P(X))$, where $\Irr(\C)$ is the (finite) set of isomorphism classes of simple objects of $\C$, and $P(X)$ denotes the projective cover of $X$.

We will need the following (straightforward) extension of \cite[Theorem 3.10]{GN} to the non-semisimple case.  

\begin{lemma}\label{posint}
Let $\C$ be a finite tensor category with simples $X_i$, projective covers $P_i$, and $\FPdim(X_i)=d_i$. Suppose $\FPdim(\C)$ is an integer. If $\Hom(P_i,P_j)\ne 0$ (i.e., $X_i$ occurs in $P_j$) then $d_id_j$ is an integer. In particular, $d_i^2$ is an integer, so $d_i$ is the square root of a positive integer.
\end{lemma} 

\begin{proof}
Let $\FPdim(\C)=d$. Let $N_{ij}$ be the multiplicity of $X_i$ in $P_j$. Then $\sum_{i,j} N_{ij}d_id_j=d$. Hence for any $g\in {\rm Gal}(\overline{\mathbb{Q}}/\mathbb{Q})$, $\sum_{i,j}N_{ij}g(d_id_j)=d$. Also, the numbers $d_i$ are algebraic integers largest in absolute value in their Galois orbits. Hence $|g(d_id_j)|\le d_id_j$ for all $i,j$. This means that $g(d_id_j)=d_id_j$, (i.e., $d_id_j$ is an integer) whenever $N_{ij}\ne 0$, as desired.
\end{proof} 

Recall that a finite tensor category $\C$ is called {\em weakly integral} if $\FPdim(\C)$ is an integer.

\begin{corollary}\label{gn} 
Let $\C$ be a weakly integral finite tensor category. Then there is an elementary abelian $2$-group $E$, a set of distinct square free  integers $n_x>0$, $x\in E$, with $n_0=1$, and a faithful grading \linebreak $\C =\oplus_{x\in E}\,\C_x$, such that $\FPdim(X)\in \mathbb{Z}\sqrt{n_x}$ for each $X\in \C_x$. 
In particular, any tensor subcategory of $\C$ is weakly integral.
\end{corollary}

\begin{proof}
By Lemma \ref{posint}, every simple object of $\C$ has dimension $\sqrt{n}$ for some positive integer $n$.
Let $\C_0 \subset \C$ be the tensor subcategory generated by all simple objects of integer FP dimension. Then all objects of $\C_0$ have integer FP dimension, since an integer can equal a sum of square roots of integers only if all the summands are integers. A similar argument shows that 
for each square free positive integer $n$ the simple objects of $\C$
whose dimension is in $\mathbb{Z}\sqrt{n}$ generate a $\C_0$-subbimodule category $\C_n$ of $\C$. Moreover, it follows from Lemma \ref{posint} that for every $n$, $\C_n\subset \C$ is a {\em Serre subcategory} (see \cite[Definition 4.14.1]{EGNO}). 

Let
$$
E:=\{n  \mbox{ is square free } \mid  \exists X\in \C,\, \mbox{ such that } \FPdim(X)\in
\mathbb{Z}\sqrt{n}\}.
$$
It is clear that for $X\in \C_n$
and $Y\in \C_m$, $X\otimes Y$ is in $\C_{(nm)'}$ where $l'$ denotes the square free part of $l$.
This defines a commutative group operation on $E$ and a faithful grading on $\C$. Since the order of every element in $E$ is
at most $2$, $E$ is an elementary abelian $2$-group.

Finally, let $\D\subset \C$ be a tensor subcategory. Since $C_x$ are Serre subcategories, for any simple $X\in \C_x$ the projective cover $P_{\D}(X)$ of $X$ in $\D$ is also in $\C_x$ since $P(X)$ surjects onto $P_{\D}(X)$ by the proof of \cite[Proposition 6.3.3]{EGNO}, so $\FPdim(X)\FPdim(P_{\D}(X))$ is an integer. 
Hence the sum of all these numbers over all the simples $X\in \D$, which is $\FPdim(\D)$, is also an integer. 
\end{proof}

Recall \cite[Section 3]{EO} that a left $\C$-module category $\N$ is called {\em indecomposable} if it is not a direct sum of
two nonzero module categories, and is called {\em exact} if $P\otimes N$ is projective for every projective $P\in \C$ and every $N\in \N$. 
The same definition applies to right module categories.

\subsection{Centralizers and Lagrangian subcategories}
\label{Centralizers and Lagrangian subcategories}
 
Let $\mathcal{C}$ be a finite {\em braided} tensor category with braiding $c$, and let $s$ be the squared braiding, i.e., 
\begin{equation}\label{squared braiding}
{\rm s}_{X,Y}:={\rm c}_{Y,X}\circ {\rm c}_{X,Y},\,\,\,X,Y\in\C.
\end{equation}
Recall that two objects $X,Y\in{\mathcal{C}}$ {\em centralize} each other if ${\rm s}_{X,Y}= \id_{X\otimes{Y}}$, and that the (M\"{u}ger) {\em centralizer} $\mathcal{D}'$ of a full tensor subcategory $\mathcal{D}\subset\mathcal{C}$ is the full tensor subcategory of $\mathcal{C}$ consisting of all objects which centralize every object of $\mathcal{D}$ (see, e.g., \cite{DGNO}). Clearly, $\mathcal{D}$ is symmetric if and only if $\mathcal{D}'\subset{\mathcal{D}}$. If $\mathcal{D}' = \mathcal{D}$ then $\mathcal{D}$ is called {\em Lagrangian}. A Lagrangian subcategory of $\mathcal{C}$ is a maximal full symmetric tensor subcategory of $\mathcal{C}$. The category $\C$ is called {\em non-degenerate} if $\mathcal{C}'=\Vect$, {\em slightly degenerate} if $\mathcal{C}'=\sVec$, and {\em degenerate} otherwise.

By \cite[Theorem 4.9]{Sh}, we have  
\begin{equation}\label{shimizu}
\FPdim(\mathcal{D})\FPdim(\mathcal{D}')=\FPdim(\mathcal{C})\FPdim(\mathcal{C}'\cap \mathcal{D})
\end{equation}
and
\begin{equation}\label{shimizu''}
\mathcal{D}''=\mathcal{D}\vee\mathcal{C}'.
\end{equation}

We will say that two objects $X,Y\in{\mathcal{C}}$ {\em anti-centralize} each other if ${\rm s}_{X,Y}= -\id_{X\otimes{Y}}$, and that the {\em anti-centralizer} of a full tensor subcategory $\mathcal{D}\subset\mathcal{C}$ is the full subcategory of $\mathcal{C}$ consisting of all objects which anti-centralize every object of $\mathcal{D}$. Note that the anti-centralizer of $\D$ is not a tensor subcategory of $\C$ (unlike $\D'$).

\subsection{Exact commutative algebras}\label{Exact commutative algebras}
Let $\C$ be a finite {\em braided} tensor category. Recall \cite[Section 8.8]{EGNO} that an algebra object $A$ in $\C$ such that $\dim(\Hom(\be ,A))=1$ is called {\em exact} if the category $\Mod(A)_{\C}$ of right $A$-modules in $\C$ is an exact indecomposable $\C$-module category. 

Now let $A$ be an exact {\em commutative} algebra object in $\C$, and let $\A:=\Mod(A)_{\C}$. Then using the braiding on $\C$ and its inverse one can define on every $M\in \A$ two structures $M_+,\,M_-$ of a {\em left} $A$-module:
$$A\otimes M_+\xrightarrow{{\rm c}_{A,M}}M_+\otimes A\xrightarrow{}M_+\,\,\,\text{and}\,\,\,A\otimes M_-\xrightarrow{c^{-1}_{M,A}}M_-\otimes A\xrightarrow{}M_-$$
(see \cite{DGNO2}, \cite[Excercise 8.8.3]{EGNO}).
Both structures turn $M$ into an $A$-bimodule, so $\A$ is fully embedded in the finite tensor category $\Bimod_{\C}(A)$, and hence inherits from it a structure of a finite tensor category with tensor product $\otimes{\rm s}_A$ \cite[Section 8.8]{EGNO}.

Recall \cite[Proposition 8.8.10]{EGNO} that the {\em free module functor} 
\begin{equation}\label{freemodfun}
F: \C \twoheadrightarrow \A,\,\,X\mapsto X\otimes A,
\end{equation}
is a surjective \footnote{I.e., any $Y\in \A$ is a subquotient of $F(X)$ for some $X\in \C$.} tensor functor.

\begin{lemma}\label{dimmodca}
We have
$$\FPdim(\A)=\frac{\FPdim(\C)}{\FPdim_{\C}(A)}.$$
\end{lemma}

\begin{proof}
Let $I:\A\to\C$ be the forgetful functor. Then $I$ is right adjoint to $F$, and $I({\bf 1})=I(A)=A$. Thus the claim follows from \cite[Lemma 6.2.4]{EGNO}.
\end{proof}
 
Recall from \cite[Proposition 8.8.10]{EGNO} that the free module functor (\ref{freemodfun}) has a {\em central structure}, i.e., it lifts to a braided tensor functor

\begin{equation}\label{centralfunc}
\widetilde{F}: \C \to \Z(\A)
\end{equation} 
in such a way that $F$ is the composition of $\widetilde{F}$ and the forgetful functor $\Z(\A)\twoheadrightarrow \A$.

The following results were proved for fusion categories in \cite[Proposition 4.2]{DGNO2} and \cite[Corollary 3.32]{DMNO}. Thanks to \cite[Theorem 1.1]{Sh}, which states that $\C$ is non-degenerate if and only if it is factorizable, and  (\ref{shimizu}), the proof in the non-semisimple case is parallel. 
 
\begin{theorem}\label{injectivefun} 
Assume $\C$ is {\em non-degenerate}. Then the following hold: 
\begin{enumerate}
\item
The braided tensor functor $\widetilde{F}:\C \to \Z(\A)$  (\ref{centralfunc}) is injective.

\item
There is a braided tensor equivalence $\C\boxtimes \C'\cong\Z(\A)$, where $\C'$ is the centralizer of $\C$ in $\Z(\A)$. In particular, $\C'$ is non-degenerate.
\end{enumerate}
\end{theorem} 
\noindent If the braided tensor functor (\ref{centralfunc}) is an equivalence, the algebra $A$ is called {\em Lagrangian}.

\subsection{Equivariantization and de-equivariantization}
\label{subsection:equivariantization and de-equivariantization}

Let $\B$ be a finite tensor category with an action of a finite group $G$. Let $\B^G$ be the {\em $G$-equivariantization} of $\B$, i.e., the category of objects in $\B$ endowed with an action of $G$. Recall that $\B^G$ is a finite tensor category containing $\Rep(G)$ as a full Tannakian subcategory, and that we have $\FPdim(\B^G)=|G|\FPdim(\B)$.
(For more details see, e.g., \cite[Section 4.15]{EGNO}.) 

Let $\C$ be a finite braided tensor category containing $\Rep(G)$ as a full Tannakian subcategory. Let $A:={\rm
Fun}(G)$ be the algebra of functions on $G$ (= regular algebra). Then $A$ is a commutative algebra in $\Rep(G)$, and hence in $\C$. Recall that the {\em
de-equivariantization} $\C_G$ of $\C$ is the finite tensor category $\Mod(A)_{\C}$ of right $A$-modules in $\C$ (see \ref{finitetensorcat}), and that $\FPdim(\C_G)=\FPdim(\C)/|G|$. (For more details see, e.g., \cite[Section 8.23]{EGNO}.)

Let $\E\subset \A$ be finite tensor categories. Recall from \cite{GNN} that the {\em relative center} $\Z_{\E}(\A)$, which is dented by $\Z(\E;\A)$ in \cite{Sh}, is the category of exact $\E$-bimodule functors from $\E$ to $\A$. By \cite[Lemma 4.5]{Sh}, $\Z_{\E}(\A)$ is a finite tensor category. The following result was proved for fusion categories in \cite[Theorem 3.5]{GNN}. The proof in the non-semisimple case is parallel.

\begin{theorem}\label{gnn} 
Let $\A=\oplus_G \A_g$ be a finite tensor category, faithfully graded by a finite group $G$, with identity component $\A_1=\E$. Then the following hold:
\begin{enumerate}
\item
There is a braided tensor equivalence $\Z(\A)\cong (\Z_{\E}(\A))^{G}$. In particular $\Z(\A)$ contains a Tannakian subcategory $\mathcal{T}:=\Rep(G)$. 

\item
The forgetful functor $\Z(\A)\twoheadrightarrow \A$ maps $\mathcal{T}$ and $\mathcal{T}'$ to $\Vect$ and $\E$, respectively. 

\item
There is a braided tensor equivalence $(\mathcal{T}')_{G}\cong \Z(\E)$. \qed
\end{enumerate}
\end{theorem}

\subsection{The symmetric tensor category $\sVec$}\label{The symmetric tensor category sVec} 
Let us first recall the definition of the finite symmetric tensor category $\sVec$. The objects of $\sVec$ are finite dimensional {\em supervector spaces} (i.e., $\mathbb{Z}/2\mathbb{Z}$-graded vector spaces $V=V_0\oplus{V_1}$), and the morphisms are {\em even} linear maps $f:V\to{W}$ between them (i.e., $f(V_i)\subset{W_i}$ for $i=0,1$). The elements in $V_0\cup V_1$ are called {\em homogeneous}, the elements in $V_0$ are called {\em even}, and the elements in $V_1$ are called {\em odd}. The {\em parity} of a homogeneous element $v$ is denoted by $|v|$. The symmetric structure of $\sVec$ is given by the collection $\{{\rm c}_{V,W}\mid\,\,\,V,W\in \sVec\}$, where ${\rm c}_{V,W}:V\otimes W\xrightarrow{\cong}W\otimes V$ is defined on homogeneous elements by 
$$v\otimes w\mapsto (-1)^{|v||w|}w\otimes v.$$

Recall that a {\em superalgebra} $A$ is an algebra in $\sVec$. In other words, $A$ is a supervector space equipped  with an even linear map $m:A\otimes{A}\to{A}$, such that $(A,m)$ is an ordinary associative algebra. Similarly, {\em supercoalgebras} and {\em Hopf superalgebras} are coalgebras and Hopf algebras in $\sVec$, respectively. 

A Hopf superalgebra $A=A_0\oplus{A_1}$ is called {\em supercommutative} if for every homogeneous elements $a,b\in{A}$, $b{a} = (-1)^{|a||b|}ab$, and is called {\em supercocommutative} if for every homogeneous element $a\in{A}$, $\Delta(a)=\sum a_1\otimes a_2= \sum (-1)^{|a_1||a_2|}a_2\otimes a_1$. 

A left {\em supermodule} over a superalgebra $A$ is a supervector space $V=V_0\oplus{V_1}$ together with an even linear map $\mu:A\otimes{V}\to{V}$ such that $(V,\mu)$ is an ordinary left module over the ordinary algebra $A$. We let $\sRep(A)$ denote the category of finite dimensional (left) supermodules over $A$ with even morphisms. Then $\sRep(A)$ is an abelian category. If moreover $A$ is a finite dimensional Hopf superalgebra then $\sRep(A)$ is a finite tensor category.

\subsection{The tensor category $\sRep(W)$}\label{section:Rep(W)}
Let $W$ be an $n$-dimensional vector space, viewed as a purely odd supervector space. Then the exterior algebra $\wedge{W}$ is a supercommutative and supercocommutative Hopf superalgebra (see \ref{The symmetric tensor category sVec}) such that $\Delta(w) = 1\otimes{w} + w\otimes{1}$, $\varepsilon(w) = 0$, and $S(w)=-w$ for every $w\in{W}$. 
It follows that $\sRep(W):=\sRep(\wedge{W})$ is a finite tensor category, which depends only on the dimension $n$ of $W$ (up to tensor equivalence). Observe also that any linear surjective map $W\twoheadrightarrow V$ induces an injective tensor functor $\sRep(V)\xrightarrow{1:1}\sRep(W)$ in a natural way. Thus, if $V$ is a quotient space of $W$ then $\sRep(V)$ can be viewed as a tensor subcategory of $\sRep(W)$.

Recall that the category $\sRep(W)$ has exactly two nonisomorphic simple objects, ${\bf 1}:=\mathbb{C}^{1|0}$ and $S:=\mathbb{C}^{0|1}$, both of which are invertible. We have $P({\bf 1})=\wedge W$, $P(S)=S\otimes P({\bf 1})$, and $$\FPdim(\sRep(W))=2\FPdim(P({\bf 1}))=2\dim(\wedge{W})=2^{n+1}.$$ 
Note that if $n$ is even then $P({\bf 1})$ and $P(S)$ are self dual, while if $n$ is odd then $P({\bf 1})$ and $P(S)$ are dual to each other.

Since $\wedge W$ is supercocommutative, $(\wedge{W},1\ot 1)$ is a triangular Hopf superalgebra. Hence, $\sRep(\wedge{W},1\ot 1)$ is a finite symmetric tensor category, which we will denote by $\E_n$ (or $\sRep(W)$, when there is no confusion). For example, we have $\E_0= \sRep(0)=\sVec$.

Recall that the tensor categories $\sRep(W)$ can be realized as representation categories of certain ordinary Hopf algebras. Namely, 
let $E(n)$ be the Nichols Hopf algebra, $n\ge 0$. Namely,  
as an algebra $E(n)=\mathbb{C}[\mathbb{Z}/2\mathbb{Z}]\ltimes \wedge W$, where $\mathbb{Z}/2\mathbb{Z}=\la u\ra$ is the group of grouplike elements of $E(n)$ and $W$ is the space of $(1,u)$-skew-primitive elements of $E(n)$. Recall \cite[Theorem 3.1.1]{AEG} that we have a {\em tensor} equivalence 
\begin{equation}\label{tensor equivalence}
\Rep(E(n))\cong\sRep(W).
\end{equation}

Let $\chi\in E(n)^*$ be the $1$-dimensional representation of $E(n)$ given by $\chi(u)=-1$ and $\chi(W)=0$. It is well known that $E(n)\xrightarrow{\cong} E(n)^*$ as Hopf algebras, and that every such isomorphism maps $u$ to $\chi$ and $W$ isomorphically onto $W^*$ (see, e.g., \cite[4.3]{BN}). In particular, $W^*$ is the space of nontrivial $(\varepsilon,\chi)$-skew-primitive elements of $E(n)^*$. Note that $\chi$ corresponds to $S$ under (\ref{tensor equivalence}).

Let $\Ext^1({\bf 1},S)$ denote the space of extensions of ${\bf 1}$ by $S$ in $\sRep(W)$. Recall that $\Ext^1({\bf 1},{\bf 1})=0$ \cite[Theorem 2.17]{EO}. Thus, $\Ext^1(S,S)=0$. Note also that $\Ext^1(S,{\bf 1})$ and $\Ext^1({\bf 1},S)$ are canonically isomorphic.

\begin{lemma}\label{ext1w*}
The following hold:
\begin{enumerate}
\item
We have a linear isomorphism $\Ext^1({\bf 1},S)\cong W^*$.
\item
There exists a short exact sequence in $\sRep(W)$ 
$$0\to W\ot S\to U\to {\bf 1}\to 0,$$
called the universal extension of ${\bf 1}$ by $S$, for which we have $\Ext^1({\bf 1},U)=0$.
\end{enumerate}
\end{lemma}

\begin{proof}
(1) By (\ref{tensor equivalence}), we have $\Ext^1({\bf 1},S)\cong \Ext^1_{\Rep(E(n))}(\varepsilon,\chi)$. Thus, the claim follows from the well known fact $\Ext^1_{\Rep(E(n))}(\varepsilon,\chi)\cong W^*$ (see, e.g., \cite[Remark 2.11]{BN}).

(2) By (1) and $\Ext^1(S,S)=0$, we have
a short exact sequence
$$0\to {\bf 1}\to U\to S^{\oplus n}\to 0$$
in $\sRep(W)$ for which $\Ext^1(S,U)=0$ (see, e.g., \cite[Lemma 4.2]{EH}). Dualizing this sequence, we get the statement.
\end{proof}

Recall that quasitriangular structures on $E(n)$ are in one to one correspondence $R_A\mapsfrom A$ with matrices $A\in M_n(\mathbb{C})$ \cite[Proposition 2, Proposition 3]{PO}, and that $(E(n),R_A)$ is triangular if and only if $A$ is symmetric \cite[Proposition 1.1]{CC}.  
In particular, the $0$ matrix corresponds to the triangular structure  
\begin{equation}\label{r0}
R_0=\frac{1}{2}\left(1\ot 1+1\ot u+u\ot 1-u\ot u\right).
\end{equation}
Recall that by \cite[Theorem 3.1.1]{AEG}, we have  
\begin{equation}\label{symmetric tensor equivalence}
\E_n\cong\Rep(E(n),R_0)
\end{equation} 
as {\em symmetric tensor} categories. Note that $\E_n$ contains $\sVec$ as a symmetric tensor subcategory for every $n$.

Recall also \cite[Corollary 2.8]{CC} that for every $A\in M_n(\mathbb{C})$, we have a braided tensor equivalence
$$\Rep(E(n),R_A)\cong\Rep(E(n),R_{A-A^T}).$$

For every $n$, let $I_n\in M_n(\mathbb{C})$ denote the identity matrix, and let $R_n$ be the $R$-matrix of $E(2n)$ corresponding to the matrix
$$
\left(\begin{array}{cc}0&I_n\\
0&0\end{array}\right)\in M_{2n}(\mathbb{C}).
$$
Let 
$$
J_n=\left(\begin{array}{cc}0&I_n\\
-I_n&0\end{array}\right)\in M_{2n}(\mathbb{C}).
$$

\begin{proposition}\label{srepvsldeg}
If ${\rm Rep}(E(2n),R)$ is slightly degenerate (see \ref{Centralizers and Lagrangian subcategories}) then 
$${\rm Rep}(E(2n),R)\cong\Rep(E(2n),R_n)$$ as braided tensor categories.
\end{proposition}

\begin{proof}
Since ${\rm Rep}(E(2n),R)$ is slightly degenerate, it follows from \cite[Theorem 2.9]{CC} that ${\rm Rep}(E(2n),R)\cong\Rep(E(2n),J_n)$ as braided tensor categories. Since $J_n=R_n-R_n^T$, the claim follows from \cite[Corollary 2.8]{CC}.
\end{proof}

For every $n$, set
\begin{equation}\label{dn}
\D_n:=\Rep(E(2n),R_n).
\end{equation}

\begin{remark}
It follows that $\D_n\cong \sVec(W\oplus W^*)$ as braided tensor categories, where the braiding on $\sVec(W\oplus W^*)$ corresponds to the canonical symplectic structure on $W\oplus W^*$.
\end{remark}

\subsection{The non-degenerate braided tensor category $\Z(\sRep(W))$}\label{section:center of Rep(W)}
Retain the notation of \ref{section:Rep(W)}. 
Recall that the center $\mathcal{Z}(\sRep(W))$ of $\sRep(W)$ is a finite braided tensor category. By (\ref{tensor equivalence}), we have 
$$\mathcal{Z}(\sRep(W))\cong\Rep(D(E(n)),\mathcal{R})$$ 
as {\em braided tensor} categories, where $D(E(n))$ is the Drinfeld double of $E(n)$ and $\mathcal{R}$ is its canonical $R$-matrix (see, e.g., \cite[4.3]{BN}).

\begin{theorem}\label{thecenterisnondeg}
For every $n$, $\Z(\E_n)$ is a finite non-degenerate braided integral tensor category containing $\E_n$ as a Lagrangian subcategory. 
\end{theorem}

\begin{proof}
By \cite[Proposition 8.6.3]{EGNO}, $\Z(\E_n)$ is factorizable. Hence by \cite[Theorem 1.1]{Sh}, $\Z(\E_n)$ is non-degenerate. Since $\E_n$ is braided, we have a canonical injective braided tensor functor $\E_n\hookrightarrow \Z(\E_n)$, so we may view $\E_n$ as a symmetric tensor subcategory of $\Z(\E_n)$.

Now by (\ref{shimizu}), $\FPdim(\E_n)\FPdim(\E_n')=\FPdim(\Z(\E_n))$. Thus, since $\FPdim(\Z(\E_n))=\FPdim(\E_n)^2$ and $\E_n\subset\E_n'$, we obtain that $\E_n=\E_n'$, as desired.
\end{proof}

\begin{proposition}\label{srepvsldegindouble}
There exists an injective braided tensor functor $$\D_n\xrightarrow{1:1} \Z(\E_n).$$
\end{proposition}

\begin{proof}
It is straightforward to verify that the map $D(E(n))\to E(2n)$ given by the identity matrix $I_{2n}$ is a surjective homomorphism of quasitriangular Hopf algebras $(D(E(n)),\mathcal{R})\twoheadrightarrow (E(2n),R_n)$.
\end{proof}

\begin{remark}
The categories $\Rep(D(E(n)))$ were recently studied by Bontea and Nikshych in \cite{BN}, where they describe their  varieties of Lagrangian subcategories.
\end{remark}

\section{The $\mathbb{Z}/2\mathbb{Z}$-faithful grading}
\label{section:proof of pre main theorem (1)}

Retain the notation of \ref{section:Rep(W)}. Set $n:=\dim(W)$, and $\E:=\E_n$.
 
Let $\mathcal{C}$ be a finite {\em non-degenerate} braided tensor category containing a Lagrangian subcategory $\mathcal{E}=\sRep(W)\subset{\mathcal{C}}$ (see \ref{Centralizers and Lagrangian subcategories}).

Let $S\in{\mathcal{E}}$ be the unique nontrivial invertible object of $\E$. We have inclusions of braided tensor categories $\sVec=\la {\bf 1},S\ra\subset\E\subset\C$. Let $\C_0:=(\sVec)'$ be the centralizer of $\sVec$ inside $\C$, and let $\C_1$ be the anti-centralizer of $\sVec$ inside $\C$ (see \ref{Centralizers and Lagrangian subcategories}). Clearly, $\E\subset\C_0$ is a symmetric tensor subcategory. Also since $\mathcal{C}$ is non-degenerate it follows that $\C_0$ is strictly contained in $\C$, and by (\ref{shimizu''}), $\C_0'=\sVec$. Thus, $\C_0$ is slightly degenerate (see \ref{Centralizers and Lagrangian subcategories}).

\begin{lemma}
\label{proposition:action of S on simple objects in C braided non-degenerate}
If $X\notin\C_0$ is simple in $\mathcal{C}$ then $X$ belongs    
to $\C_1$. \footnote{The assumption that $\mathcal{C}$ is non-degenerate is not needed in this proposition.}
\end{lemma}

\begin{proof}
By Schur's lemma, ${\rm s}_{S,X}=\lambda\cdot \id_{S\otimes X}$ for some $\lambda\in \mathbb{C}^{\times}$ (see (\ref{squared braiding})). Hence, we have 
\begin{eqnarray*}
\lefteqn{\id_X={\rm s}_{S\otimes S ,X}={\rm c}_{X,S\otimes S} \circ {\rm c}_{S\otimes S ,X}}\\ 
& = & ((\id_S\otimes {\rm c}_{X,S})\circ ({\rm c}_{X,S}\otimes \id_S))\circ (({\rm c}_{S,X}\otimes \id_S)\circ (\id_S\otimes {\rm c}_{S,X}))\\ 
& = & \lambda^2\cdot \id_{S\otimes S\otimes X}=\lambda^2\cdot \id_X,
\end{eqnarray*}
which implies that $\lambda^2=1$. Since $X\notin \C_0=(\sVec)'$, it follows that $\lambda=-1$, as claimed.
\end{proof}

It follows from Lemma \ref{proposition:action of S on simple objects in C braided non-degenerate} that $\C_1\ne 0$ is a full abelian subcategory of $\C$, and that $\C$ admits a $\mathbb{Z}/2\mathbb{Z}$-faithful grading $$\mathcal{C}=\C_0\oplus \C_1,$$ with $\C_0$ being the identity component. In particular, $\C_1$ is an invertible $\C_0$-bimodule category of order $2$, $\FPdim(\mathcal{C}) = 2\FPdim(\C_0)$, and the projective covers  $P({\textbf{1}})$ and $P(S)$ belong to $\C_0$. 

\begin{proposition}\label{nosimple}
The objects ${\textbf{1}},S$ are the unique simple objects of $\C_0$. In particular, $\C_0$ is a finite {\em pointed} tensor category \footnote{I.e., all its simple objects are invertible.}.
\end{proposition}

\begin{proof} \footnote{We are grateful to Pavel Etingof for help with the proof.}
The statement is clear for $n=0$. 

Assume $n=1$, and suppose $\C_0$ has a simple object not isomorphic to ${\textbf{1}}$ or $S$. Let $\D\subset \C_0$ be the Serre tensor closure of $\E$. Note that by our assumption, $\D$ is strictly contained in $\C_0$. By Corollary \ref{gn}, $\D$ has an integer FP dimension (as $\FPdim(\C_0)=8$). Since $\D$ contains $\E$, we have $4\le \FPdim(\D)<8$, and hence since $\FPdim(\D)$ divides $8$, we have $\D=\E$. This implies that the projective cover $P(\textbf{1})$ in $\C_0$ (and hence in $\C$) coincides with the projective cover $P_{\E}(\textbf{1})$ of $\textbf{1}$ in $\E$. But $\C$ (and hence $\C_0$) is unimodular by \cite[Proposition 8.10.10]{EGNO} and \cite[Theorem 1.1]{Sh}, while $\E$ is not, a contradiction.

{\bf From now on} we assume that $n\ge 2$.

Since the corresponding Nichols Hopf algebra $E(n)$ is generated in degree $1$ (see \ref{section:Rep(W)}), it follows from (\ref{tensor equivalence}) and 
\cite[Proposition 5.11.9]{EGNO} that $\E$ is tensor generated by indecomposable $2$-dimensional objects. 
 
Let $V$ be a simple object of $\C_0$ not isomorphic to ${\textbf{1}}$ or $S$. Then there exists a nontrivial extension $E$ of ${\textbf{1}}$ by $S$ in $\E$ such that $V$ does not centralize $E$. Indeed, otherwise $V$ would centralize $\E$ (as it is generated by indecomposable $2$-dimensional objects, which are all extensions of ${\textbf{1}}$ by $S$ and their duals), so $V\in \E$, a contradiction. 
 
This implies that $S\otimes V\cong V$. Indeed, otherwise $E\otimes V$ has a $2$-step filtration with quotients $V',V$ with $V'\ne V$. Let $s:={\rm s}_{E,V}$. Then $s-\id$ is strictly upper triangular with respect to the above filtration, so it maps $V$ to $V'$, i.e., is zero, a contradiction. 

Now let 
$$0\to W\otimes S\to U\to \textbf{1}\to 0$$ be the universal extension of ${\textbf{1}}$ by a multiple of $S$ in $\E$ (see Lemma \ref{ext1w*}). Consider the endomorphism $s-\id$ of $U\otimes V$, where $s:={\rm s}_{U,V}$. The object $U\otimes V$ has a $2$-step filtration with quotients $W\otimes V$, $V$, and $s-\id$ is strictly upper triangular under this filtration, i.e., defines a morphism $V\to W\otimes V$, i.e., a vector $w$ in $W$. This  
vector is well defined up to scaling, since it rescales when we rescale the isomorphism $V\xrightarrow{\cong}S\otimes V$. Also $w\ne 0$, since all extensions of $\textbf{1}$ by $S$ are quotients of $U$, and there exists one not centralizing with $V$ as shown above. Thus, we obtain a well defined line $L_V$ in $W$ spanned by $w$. 

Since the number of simple objects $V$ is finite, there exist distinct codimension $1$ subspaces $W_1, W_2$ in $W$ which do not contain $L_V$ for any $V$. Consider the subcategories $\D_i:=\sRep(W/W_i)$ in $\sRep(W)$ (see \ref{section:Rep(W)}), $i=1,2$, and let $\D_i'$ be the centralizers of $\D_i$ inside $\C$ (equivalently, inside $\C_0$). Then $\D_i'$ cannot contain any simple object $V$ not isomorphic to $\textbf{1},S$ (as $V$ does not centralize $\D_i$, since $L_V$ is not contained in $W_i$). Since $\FPdim(\D_i)=4$, $i=1,2$, we have $\FPdim(\D_i')=\FPdim(\C_0)/2$ for each $i$. Consider the tensor subcategory $\D$ of $\C_0$ generated by $\D_1'$ and $\D_2'$. It has integer FP dimension by Corollary \ref{gn}, which divides $\FPdim(\C_0)$ (by \cite[Theorem 3.47]{EO}). Also $\D$ is bigger than $\D_1', \D_2'$ (as $\D_1\ne \D_2$ as subcategories of $\C_0$), hence $\FPdim(\D)>\FPdim(\C_0)/2$. Hence $\FPdim(\D)=\FPdim(\C_0)$, i.e., $\D=\C_0$. But $\D$ has no simple objects other than $\textbf{1},S$ (as $\D_1', \D_2'$ do not have such objects). The proposition is proved.  
\end{proof}

\begin{corollary}
\label{proposition:projective covers of simple one dimensional objects of C generate a full tensor subcategory}
We have a braided tensor equivalence $\C_0\cong\D_n$ (see (\ref{dn})).
\end{corollary}

\begin{proof} 
By Proposition \ref{nosimple}, $\C_0$ is a finite pointed tensor category with semisimple part $\la{\bf 1},S\ra\cong\Rep(\mathbb{Z}/2\mathbb{Z})$ of FP dimension $2$ with the {\em trivial} associativity. Hence by \cite[Theorem 3.1]{EG}, $\C_0$ is tensor equivalent to $\Rep(E(m))$ for some $m$.

Moreover, on one hand, we have $$\FPdim(\C_0)=2\FPdim(P({\bf 1}))=2^{m+1}$$ and, on the other hand, we have $$\FPdim(\C_0)=\FPdim(\C)/2=\FPdim(\E)^2/2=2^{2n+1}.$$ Hence, $m=2n$.

Thus, since $\C_0$ is slightly degenerate, the statement follows from Proposition \ref{srepvsldeg}.
\end{proof}

Since the projective objects of $\C_0$ are projective in $\mathcal{C}$ (as they are direct sums of copies of  $P({\textbf{1}})$ and $P(S)$), $\mathcal{C}$ and $\C_1$ are exact module categories over $\C_0$ via the tensor product in $\C$.

\begin{proposition}
\label{proposition:C braided non-degenerate has four simple objects}
The exact $\C_0$-module category 
$\C_1$ is indecomposable, with at most two non-isomorphic simple objects.
\end{proposition}

\begin{proof}
By Corollary \ref{proposition:projective covers of simple one dimensional objects of C generate a full tensor subcategory}, $\mathcal{C}$ has more than two simple objects. If $\C_1$ has exactly one simple object, there is nothing to prove. Let us therefore assume that $\C_1$ has {\bf at least two} nonisomorphic simple objects. 

Decompose $\C_1$ into a direct sum $\C_1= \oplus_{i = 1}^n\mathcal{M}_i$ of exact indecomposable module categories over $\C_0$. By Proposition \ref{proposition:projective covers of simple one dimensional objects of C generate a full tensor subcategory} and \cite[Example 4.7]{EO}, each module category $\mathcal{M}_i$ has at most two simple objects.

Suppose that $\mathcal{M}_1,~\mathcal{M}_2\ne{0}$, and let $X\in\mathcal{M}_1,~Y\in\mathcal{M}_2$ be simple objects. Since $\mathcal{C}$ is generated by $X$ as a module category over itself there exists an object $Z\in\mathcal{C}$ such that $\Hom_{\C}(Z\otimes{X},Y) \ne{0}$. Clearly $Z\otimes{X}\notin{\mathcal{M}_1}$, and hence $Z\notin{\C_0}$. Also, since $$\Hom_{\C}(X\otimes{}^* Y,{}^* Z)\cong\Hom_{\C}(Z\otimes{X},Y) \ne{0},$$ we can choose a nonzero morphism $g:X\otimes{}^* Y\to{}^* Z$. Let $W$ be a simple quotient of $\text{Im}(g)$. Then $\Hom_{\C}(W^*\otimes{X},Y)\cong\Hom_{\C}(X\otimes{}^* Y,W)\ne{0}$. Thus we may assume that $Z$ is simple (replacing it by $W^*$, if necessary). But then it follows that $Z\in\C_1$, and hence $Z\otimes X\in\C_0$, a contradiction. Hence, $\C_1$ is indecomposable. Since by Corollary 
\ref{proposition:projective covers of simple one dimensional objects of C generate a full tensor subcategory} and \cite[Example 4.7]{EO}, $\C_1$ has at most two simple objects it follows that $\C_1$ has exactly two nonisomorphic simple objects, as desired.
\end{proof}

To summarize, we have proved the following theorem. 

\begin{theorem}
\label{pre main theorem}
Let $\mathcal{C}$ be a finite non-degenerate braided tensor category containing $\E_n$ as a Lagrangian subcategory. 
Let $\C_0$, $\C_1$ be the centralizer and anti-centralizer of $\sVec\subset\E_n$ inside $\C$, respectively. Then $\C$ admits a $\mathbb{Z}/2\mathbb{Z}$-faithful grading $\mathcal{C}=\C_0\oplus \C_1$, with $\C_0$ being the identity component, such that the following hold:
\begin{enumerate}
\item $\C_0$ contains $\E_n$ as a Lagrangian subcategory, it is slightly degenerate, and $\C_0\cong\D_n$ as braided tensor categories. 
\item $\C_1$ is an exact indecomposable $\C_0$-module category, with at most two nonisomorphic simple objects. \qed
\end{enumerate}
\end{theorem}

\begin{remark}
In \cite[Section 8]{BN2} the authors show in particular that $\Z(\E_n)$ is nilpotent of nilpotency class $2$.
\end{remark}

\begin{example}\label{minze}
Let $\C:=\Z(\sVec)$. Then $\C_0=\sVec$ and $\C_1$ has exactly two invertible objects of order $2$.

Let $\E:=\E_n$ be non-fusion (i.e., $n>0$), and let $\C:=\Z(\E)$ (see Theorem \ref{thecenterisnondeg}). We claim that the following hold:
\begin{enumerate}

\item 
$\C$ has exactly two nonisomorphic simple projective objects $p$ and $q$. In other words, $\C_1=\la p,q\ra$ is semisimple of rank $2$.

\item
$p\otimes{p^*} = P({\textbf{1}})$, $q\cong p\otimes S$, and $\FPdim(p)=\FPdim(\E)/2$.

\item 
$p$ and $q$ are self dual if and only if $\dim(W)$ is even.
\end{enumerate}

Indeed, let $A\in\C_0$ be the commutative algebra corresponding to the forgetful functor $\C_0\to \E$. It is easy to see that the Lagrangian algebra $B\in\Z(\E)$, corresponding to the forgetful functor $F:\Z(\E)\to \E$, has the form $B=A\oplus p$, for some $p\in \C_1$. In particular, we have $$\FPdim(p)=\FPdim(A)=\FPdim(\E)/2.$$ Now, since $F(P(p))$ is projective in $\E$, and $$\Hom_{\E}(F(P(p)),\textbf{1})=\Hom_{\Z(\E)}(P(p),B)=\Hom_{\Z(\E)}(P(p),p)\ne 0,$$
it follows that $P_{\E}(\textbf{1})$ projects onto $F(P(p))$. Hence, we have
\begin{eqnarray*}
\lefteqn{\FPdim(\E)/2=\FPdim(p)}\\
& \le & \FPdim(P(p))\le \FPdim(P_{\E}(\textbf{1}))=\FPdim(\E)/2,
\end{eqnarray*}
which implies that $P(p)=p$. Thus, $p$ is projective. Therefore, $F(p)$ is projective in $\E$. Since $$\Hom_{\E}(F(p),\textbf{1})=\Hom_{\Z(\E)}(p,B)=\Hom_{\Z(\E)}(p,p)\ne 0,$$
it follows that if the dimension of $\Hom_{\Z(\E)}(p,p)$ was $\ge 2$, so would be the dimension of $\Hom_{\E}(P_{\E}(\textbf{1}),\textbf{1})$ (as $P_{\E}(\textbf{1})$ projects onto $F(p)$), which is not the case. Hence, $p$ is also simple.

Now, since $p\otimes{p^*}\in\C_0$ is projective and $\Hom_{\C}(p\otimes{p^*},{\textbf{1}})$ is $1$-dimensional, we have $p\otimes{p^*}\cong P({\textbf{1}})$. In particular, it follows that $\Hom_{\C}(p,p\otimes S)=\Hom_{\C}(p\otimes{p^*},S)=0$. Hence, $q:=p\otimes S\ncong p$ is another simple projective object in $\C_1$, as desired. 

Moreover, we have
$$2\FPdim(\C_0)=\FPdim(\mathcal{C})=\FPdim(\C_0) + 2\FPdim(p)^2,$$
which implies that $\FPdim(p)^2=\FPdim(\C_0)/2=\FPdim(\E)^2/4$. We have thus established (1) and (2).

Finally, by \cite[Theorem 6.1.16]{EGNO} the forgetful functor $\Z(\E)\to\E$ maps $p$ and $q$ to projective objects, and since the Frobenius-Perron dimensions of $p$ and $q$ are  equal to $\FPdim(P_{\E}({\bf 1}))= \FPdim(P_{\E}(S))$, it follows that 
$F$ must map $p$ to $P_{\E}({\bf 1})$ or $P_{\E}(S)$, and vise versa for $q$. Thus (3) from the fact that $P_{\E}({\bf 1})$ and $P_{\E}(S)$ are self dual if and only if $\dim(W)$ is even (equivalently, $P_{\E}({\bf 1})$ and $P_{\E}(S)$ are dual to each other if and only if $\dim(W)$ is odd).
\end{example}

\section{The action of the group ${\bf B}$}\label{MissVec}

Recall from \cite[Lemma A.11]{DGNO} that there are exactly $8$ non-degenerate braided {\em pointed} fusion categories of Frobenius-Perron dimension $4$, which contain $\sVec$ as a Lagrangian subcategory (up to {\em braided tensor} equivalence), and that $4$ of them are supported on the group $\mathbb{Z}/2\mathbb{Z}\times \mathbb{Z}/2\mathbb{Z}$, while the other $4$ are supported on the group $\mathbb{Z}/4\mathbb{Z}$.

Recall also from \cite[Corollary B.16]{DGNO} that there are exactly $8$ non-degenerate braided {\em non-integral} fusion categories of Frobenius-Perron dimension $4$, which contain $\sVec$ as a Lagrangian subcategory (up to {\em braided tensor} equivalence). These non-integral categories are called {\em Ising categories}.
 
Thus, all together, there are exactly $16$ non-degenerate braided fusion categories of Frobenius-Perron dimension $4$, which contain $\sVec$ as a Lagrangian subcategory (up to {\em braided tensor} equivalence). Let us denote this set by ${\bf B}$.

Now, let $\mathcal{D}_1$ and $\mathcal{D}_2$ be two finite non-degenerate braided tensor categories containing $\sVec(W_1)$ and $\sVec(W_2)$ as a Lagrangian subcategory, respectively. Consider the finite non-degenerate braided tensor category $\mathcal{D}_1\boxtimes \mathcal{D}_2$. Then $\mathcal{D}_1\boxtimes \mathcal{D}_2$ contains $\sVec\boxtimes \sVec$ as a braided tensor category, and hence contains a Tannakian subcategory $\mathcal{T}:=\Rep(\mathbb{Z}/2\mathbb{Z})$. Let $\mathcal{T}'$ be the centralizer of $\mathcal{T}$ inside $\mathcal{D}_1\boxtimes \mathcal{D}_2$, and let $$\mathcal{D}_1\widetilde{\boxtimes} \mathcal{D}_2:=(\mathcal{T}')_{\mathbb{Z}/2\mathbb{Z}}$$
be the de-equivariantization tenosr category (see \ref{subsection:equivariantization and de-equivariantization}).
\begin{proposition}\label{actiononthecenter}
Let $\mathcal{C}$, $\mathcal{D}_1$ and $\mathcal{D}_2$ be finite non-degenerate braided tensor categories containing a Lagrangian subcategory $\E:=\sRep(W)$, $\E_1:=\sRep(W_1)$ and $\E_2:=\sRep(W_2)$, respectively. Then the following hold:
\begin{enumerate}
\item
$\mathcal{D}_1\widetilde{\boxtimes} \mathcal{D}_2$ is a finite non-degenerate braided tensor category containing $\E_3:=\sVec(W_1\oplus W_2)$ as a Lagrangian subcategory.

\item 
For every $\B\in {\bf B}$, $\C\widetilde{\boxtimes} \B$ is a finite non-degenerate braided tensor category containing $\E$ as a Lagrangian subcategory.
\end{enumerate}
\end{proposition}

\begin{proof}
(1) The first claim follows from  (\ref{shimizu''}). As for the second claim, it is clear that $\mathcal{D}_1\widetilde{\boxtimes} \mathcal{D}_2$ contains the centralizer of $\mathcal{T}$ inside $\E_1\boxtimes \E_2$, and that the later category is braided tensor equivalent to $\E_3$. The fact that $\E_3$ is Lagrangian in $\mathcal{D}_1\widetilde{\boxtimes} \mathcal{D}_2$ follows from dimension considerations.

(2) Follows from (1).
\end{proof}

In particular for $W=0$, Proposition \ref{actiononthecenter}(2) says that ${\bf B}$ forms a group under the product $\widetilde{\boxtimes}$, with unit element $\Z(\sVec)$. It is well known that ${\bf B}\cong \mathbb{Z}/16\mathbb{Z}$ \cite{DNO} (see also \cite{BGHNPRW,LKW}). Also, the pointed categories in ${\bf B}$ form a subgroup of index $2$.

Moreover, Proposition \ref{actiononthecenter}(1) states that ${\bf B}\cong \mathbb{Z}/16\mathbb{Z}$ acts on the set of equivalence classes of finite non-degenerate braided tensor categories containing a Lagrangian subcategory $\E$ of the same Frobenius-Perron dimension. Hence, by Proposition \ref{thecenterisnondeg}, 
we have the following result.

\begin{corollary}\label{actiononthecenter1}
For every $\B\in {\bf B}$, $\Z(\E)\widetilde{\boxtimes} \B$ is a finite non-degenerate braided tensor category containing $\E$ as a Lagrangian subcategory. \qed
\end{corollary}

Now let $A:=\Fun(\mathbb{Z}/2\mathbb{Z})$ be the regular algebra (see \ref{subsection:equivariantization and de-equivariantization}). For $\B\in {\bf B}$, let $F:\mathcal{T}'\twoheadrightarrow\Z(\E)\widetilde{\boxtimes}\B$, $Z\mapsto Z\otimes A$, be the free module functor  (see (\ref{freemodfun})). Recall that $F$ is a surjective braided tensor functor.

\begin{theorem}\label{simproj}
Let $\B\in {\bf B}$, and set $\C:=\Z(\E)\widetilde{\boxtimes}\B$. The following hold:
\begin{enumerate}

\item 
If $\B$ {\em is not} pointed then $\C$ {\em is not} integral, and $\C$ has precisely $2$ invertible objects ${\textbf{1}}$, $S$ (non-projective, if $W\ne 0$) and $1$ non-invertible simple projective object $X$, such that 
$$X\cong X^*\cong S\otimes X,\,\,\,X\otimes{X} = P({\textbf{1}})\oplus P(S)$$ and $\FPdim(X)=\FPdim(\E)/\sqrt{2}$.

\item 
If $\B$ {\em is} pointed then $\C$ {\em is} integral, and $\C$ has precisely $2$ invertible objects ${\textbf{1}}$, $S$ (non-projective, if $W\ne 0$) and $2$ simple projective objects $P$, $Q\cong P\otimes S$ (non-invertible, if $W\ne 0$), such that $P\otimes{P^*} = P({\textbf{1}})$ and $\FPdim(P)=\FPdim(\E)/2$.
\end{enumerate}
\end{theorem}

\begin{proof}
Let $p$, $q$ be the simple projectives of $\Z(\E)$ (see Example \ref{minze}).

(1) Let $Z$ be the unique noninvertible simple object of $\B$. We have $Z\cong Z^*\cong S\otimes Z$ and $Z\otimes Z= \textbf{1}\oplus S$.

By Theorem \ref{pre main theorem}, $p\boxtimes Z$ is in $\mathcal{T}'$, and we have that $X:=F(p\boxtimes Z)$ is simple projective in $\C$. Thus the claim follows from Example \ref{minze} and the properties of $F$.

(2) Let $h\in\B$ be as in the proof of Theorem \ref{freeaction}. By Theorem \ref{pre main theorem}, $p\boxtimes h$ and $q\boxtimes h$ are in $\mathcal{T}'$, and we have that $P:=F(p\boxtimes h)$ and $Q:=F(q\boxtimes h)$ are simple projective in $\C$. Thus the claim follows from Example \ref{minze} and the properties of $F$.
\end{proof}

We conclude this section by proving that the action of ${\bf B}$ is free.

\begin{theorem}\label{freeaction} 
The action of ${\bf B}$ on the set of equivalence classes of finite non-degenerate braided tensor categories containing $\E$ as a Lagrangian subcategory, is free.
\end{theorem}

\begin{proof} 
Fix a finite non-degenerate braided tensor category $\C$ containing $\E$ as a Lagrangian subcategory. 
It is sufficient to show that the categories $\C\widetilde{\boxtimes}\B$, $\B\in{\bf B}$ is pointed, are pairwise non-equivalent braided tensor categories.

Let $g\in\B$ be such that $S\boxtimes g$ is the nontrivial object of $\mathcal{T}$. Then $g$ has order $2$, and the braiding $c(g,g)$ on $g^2=1$ in $\B$ is equal to $-\id_{1}$. Let $A:={\bf 1}\boxtimes 1 \oplus S\boxtimes g$ be the regular algebra of $\mathcal{T}$. Also, pick a simple object $h\ne 1,g$ in $\B$. We have $c(h,g)c(g,h)=-\id_{gh}$. 

Pick an object $Z\in\C$ such that ${\rm s}_{S,Z}=-\id_{S\otimes Z}$ (such an object exists by Theorem 
\ref{pre main theorem}). Then $Z\boxtimes h\in\mathcal{T}'$. Consider the free braided tensor functor $F:\mathcal{T}'\to\C\widetilde{\boxtimes}\B,$ and let $z:=F(Z\boxtimes h)=(Z\boxtimes h)\otimes A$. Since $F$ is braided, we have the following commutative diagram
$$
\begin{CD}
F(Z\boxtimes h)\otimes_A F(Z\boxtimes h)@>\tilde{c}_{z,z}>>F(Z\boxtimes h)\otimes_A F(Z\boxtimes h)\\
@V\cong VV @V\cong VV\\
F((Z\otimes Z)\boxtimes h^2)@>F({\rm c}_{Z\boxtimes h,Z\boxtimes h})>>F((Z\otimes Z)\boxtimes h^2),
\end{CD}
$$
where $\tilde{c}_{z,z}$ is the braiding on $z\otimes_A z$ in $\C\widetilde{\boxtimes}\B$. But, 
$${\rm c}_{Z\boxtimes h,Z\boxtimes h}={\rm c}_{Z,Z}\boxtimes c(h,h)\id_{h^2},$$ which implies that $c(h,h)$ is determined by the braided tensor category $\C\widetilde{\boxtimes}\B$. Since $c(h,h)$ determines $\B$ (see, e.g., \cite[Lemma A.11]{DGNO}), we are done.
\end{proof}

\section{The proof of Theorem \ref{main theorem}}
\label{section:proof of case (3)}

Retain the notation of \ref{section:Rep(W)}, \ref{MissVec}. Set $n:=\dim(W)$, and $\E:=\E_n$.

We already proved in Theorem \ref{freeaction} that the action of ${\bf B}$ is free, so it remains to show it is transitive.
 
Let $\C=\C_0\oplus\C_1$ be a finite non-degenerate braided tensor category containing $\E$ as a Lagrangian subcategory. 
By Proposition \ref{srepvsldegindouble} and Theorem \ref{pre main theorem}, we have an injective braided tensor functor $\C_0\xrightarrow{1:1}\mathcal{Z}(\E)$. 
Composing this functor with the forgetfull functor $\mathcal{Z}(\E)\twoheadrightarrow\E$, yields a surjective {\em tensor} functor  
\begin{equation}\label{the functor F}
{\rm F}:\C_0\twoheadrightarrow \E.\footnote{If $W\ne 0$, $F$ is not braided!}
\end{equation}
It is clear that the functor $\widetilde{{\rm F}}:\C_0\twoheadrightarrow \Z(\E)$ given by the composition
$$\C_0\xrightarrow{{\rm F}} \E\hookrightarrow \Z(\E)$$
determines a central structure on ${\rm F}$ (\cite[Definition 8.8.6]{EGNO}).
  
Let ${\rm I}$ be the right adjoint functor to ${\rm F}$, and let $A:={\rm I}({\mathbf 1})$. Then $A$ has a canonical structure of an associative algebra object in $\C_0$ \cite[Example 7.9.10]{EGNO}, and $\FPdim(A)=\FPdim(\E)/2$ (see Lemma \ref{dimmodca}).

\begin{lemma}\label{A is commutative}
The algebra object $A$ is commutative, that is, we have $m=m\circ {\rm c}_{A,A}$, where $m:A\otimes A\to A$ is the multiplication map on $A$.
\end{lemma}

\begin{proof}
Follows from \cite[Proposition 8.8.8]{EGNO} since ${\rm F}$ has a central structure. 
\end{proof}

The functor ${\rm F}$ (\ref{the functor F}) defines on $\E$ a structure of an {\em exact} left indecomposable module category over $\C_0$, and by \cite[Theorem 7.10.1]{EGNO}, the functor ${\rm I}$ induces an equivalence ${\rm I}:\E\xrightarrow{\cong}\Mod(A)_{\C_0}$ of left module categories over $\C_0$.

\begin{lemma}\label{exact1}
The functor ${\rm I}:\E\xrightarrow{\cong}\Mod(A)_{\C_0}$ induces an equivalence of tensor categories,  
where $\Mod(A)_{\C_0}$ is viewed as a tensor subcategory of the finite tensor category $\Bimod_{\C_0}(A)$ (see \ref{Exact commutative algebras}).
\end{lemma}

\begin{proof}
One shows that ${\rm I}$ has a structure of a tensor functor in exactly the same way as one shows that the right adjoint to the forgetful functor $\Z(\E)\twoheadrightarrow \E$ does (see \cite[Lemma 8.12.2]{EGNO}).
\end{proof}

Since $\C_1$ is an exact invertible $\C_0$-bimodule category, the category $\Mod(A)_{\C_1}=\C_1\boxtimes _{\C_0} \Mod(A)_{\C_0}$ of right $A$-modules in $\C_1$ is naturally a left indecomposable $\C_0$-module category. 

\begin{lemma}\label{exact}
The following hold:
\begin{enumerate}
\item
There is an equivalence $\Mod(A)_{\C_1}\cong \Mod(A)_{\C_0}$ of left module categories over $\C_0$. Thus, $\Mod(A)_{\C_1}$ is exact over $\C_0$.

\item
$\Mod(A)_{\C}$ is an {\em exact} left indecomposable module category over $\C$. Hence, $\A:=\Mod(A)_{\C}$ is a finite tensor subcategory of $\Bimod_{\C}(A)$ (see \ref{Exact commutative algebras}).

\item
The tensor category $\A$ has a $\mathbb{Z}/2\mathbb{Z}$-faithful grading $\A=\A_0\oplus \A_1$, where $\A_0\cong \E$ as tensor categories and $\A_1\subset \Bimod_{\C_1}(A)$.
\end{enumerate}
\end{lemma}

\begin{proof}
(1), (2) Clearly, $\Mod(A)_{\C}\cong\C\boxtimes _{\C_0} \Mod(A)_{\C_0}$ is the induced left module category of $\Mod(A)_{\C_0}$. Thus by Theorem \ref{pre main theorem}(1), we have $$\Mod(A)_{\C}\cong(\C_0\oplus \C_1)\boxtimes _{\C_0} \Mod(A)_{\C_0}\cong \Mod(A)_{\C_0}\oplus \Mod(A)_{\C_1},$$ as left $\C_0$-module categories. It follows that $\Mod(A)_{\C_1}$ and $\Mod(A)_{\C_0}$ must be equivalent as left module categories over $\C_0$ since $\Mod(A)_{\C}$ is indecomposable over $\C$. This implies that $\Mod(A)_{\C_1}$ is exact over $\C_0$, and hence so is $\Mod(A)_{\C}$. Therefore by \cite[Corollary 2.5]{EG}, $\Mod(A)_{\C}$ is exact over $\C$. Finally, since $A$ is commutative in $\C$, $\A$ is a finite tensor subcategory of $\Bimod_{\C}(A)$ (see \ref{Exact commutative algebras}).

(3) The decomposition $\Mod(A)_{\C}=\Mod(A)_{\C_0}\oplus \Mod(A)_{\C_1}$ of module categories over $\C_0$ obtained above clearly induces a $\mathbb{Z}/2\mathbb{Z}$-grading on $\A$ with the claimed properties.
\end{proof}

It follows from Proposition \ref{injectivefun} and Lemma \ref{exact}(2) that we have a decomposition 
\begin{equation}\label{decomposition}
\C\boxtimes \C'\cong \Z(\A)
\end{equation}
of braided tensor categories, where $\C'$ is the centralizer of $\C$ inside $\Z(\A)$. Since by (\ref{decomposition}), $\C'$ is non-degenerate with Frobenius-Perron dimension $4$, $\C'$ is fusion. We have thus obtained the following.

\begin{proposition}\label{centralizer of C}
There exists an element $\B$ in the group ${\bf B}$ such that $\C'\cong\B$ as braided tensor categories. Hence, there is a braided tensor equivalence $\C\boxtimes \B\cong \Z(\A)$. \qed
\end{proposition}

Finally, it follows from Theorem \ref{gnn} and Proposition  \ref{centralizer of C} that there is a braided tensor equivalence $\C\boxtimes \B\cong (\Z_{\E}(\A))^{\mathbb{Z}/2\mathbb{Z}}$. In particular $\C\boxtimes \B$ contains a Tannakian subcategory $\mathcal{T}:=\Rep(\mathbb{Z}/2\mathbb{Z})$, and there is a braided tensor equivalence $(\mathcal{T}')_{\mathbb{Z}/2\mathbb{Z}}\cong \Z(\E)$. 
Hence $\C\widetilde{\boxtimes} \B\cong \Z(\E)$, so $\C\cong \Z(\E)\widetilde{\boxtimes} \B ^{-1}$, as desired.

The proof of Theorem \ref{main theorem} is complete. \qed

As a corollary of Theorems \ref{main theorem}, \ref{simproj} we obtain the following result.

\begin{corollary}\label{simprojex} 
Let $\C$ be finite non-degenerate braided tensor category containing a Lagrangian subcategory $\E_n$. Then the following hold:
\begin{enumerate}
\item
If $\C$ is integral then $\C$ has exactly four nonisomorphic  simple objects: two invertible objects (non-projective, if $n>0$), and two simple projective objects (non-invertible, if $n>0$).

\item
If $\C$ is not integral then $\C$ has exactly three nonisomorphic simple objects: two invertible objects (non-projective, if $n>0$), and one simple projective object. \qed
\end{enumerate}
\end{corollary}

\begin{remark}\label{ribbon}
\begin{enumerate}
\item
Some (but not all) of the finite non-degenerate braided {\em integral} tensor categories containing a Lagrangian subcategory $\E_n$ can be constructed using the interesting method developed by Davydov-Runkel in \cite{DR1,DR}. For example, $\Z(\E_n)$ can be constructed using Davydov-Runkel's method if and only if $n$ is even (see, \cite[Theorem 1.2]{DR2}).

\item
In \cite[Section 5.1]{DR1} it is shown that all $8$ equivalence classes of non-degenerate braided {\em non-integral} tensor categories containing a Lagrangian subcategory $\E_0=\sVec$ (i.e., Ising categories; see \ref{MissVec}) can be constructed using the method of Davydov-Runkel. 

In the Appendix below we use Theorem \ref{main theorem}, and the classification of $R$-matrices of the $8$-dimensional Nicols' Hopf algebra $E(2)$ given in \cite{G}, to verify that all $8$ equivalence classes of non-degenerate braided {\em non-integral} tensor categories containing a Lagrangian subcategory $\E_1$ arise in this way as well.

More generally, we expect that Theorem \ref{main theorem}, and the classification of $R$-matrices of the $2n$-dimensional Nicols' Hopf algebra $E(2n)$ given in \cite{PO}, can be used in a similar way to verify that all $8$ equivalence classes of non-degenerate braided {\em non-integral} tensor categories containing a Lagrangian subcategory $\E_n$ arise in this way for every $n$.
\end{enumerate}
\end{remark}
  
\section{Degenerate braided tensor categories with Lagrangian $\sRep(W)$}
\label{section:proof of case (2)}

Let $W$ be an $n$-dimensional vector space viewed as a purely odd supervector space. Let $\E_n=\sRep(W)$, and let $S$ be the unique nontrivial invertible object of $\E_n$.

\begin{theorem}
\label{pre main theorem1}
The following hold:
\begin{enumerate}
\item 
Assume $\mathcal{D}$ is a finite {\em degenerate} braided tensor category containing $\E_n$ as a Lagrangian subcategory. Then 
there exists a vector space $V$, with $n\le{\dim(V)}\le 2n$, such that $\mathcal{D}\cong \sRep(V)$ as tensor categories, and we have $\FPdim(\D')=2^{2n-\dim(V)+1}$. 

\item For every integer $\ell$ such that $n\le \ell\le 2n$, there exist a $\ell$-dimensional vector space $V$ and a finite {\em degenerate} braided tensor category $\mathcal{D}$ containing $\E_n$ as a Lagrangian subcategory, such that $\FPdim(\D')=2^{2n-\ell+1}$.
\end{enumerate} 
\end{theorem}

\begin{proof} 
(1) We have $\Vect\ne \mathcal{D}'\subset \mathcal{E}_n'=\mathcal{E}_n$, hence $S\in\mathcal{D}'$. Arguing now as in the proof of Proposition \ref{nosimple}, we conclude that ${\bf 1},S$ are the only simple objects of $\mathcal{D}$. Therefore it follows from \cite[Theorem 3.1]{EG} that $\mathcal{D}$ is tensor equivalent to $\sRep(V)$ for some finite dimensional purely odd supervector space $V$.

Since $\sRep(W)\subset{\sRep(V)}$ as tensor categories, $n\le{\dim(V)}$. Moreover, by (\ref{shimizu}),
\begin{equation*}
\FPdim(\mathcal{D}){\FPdim(\mathcal{D}')}=\FPdim(\mathcal{E}_n){\FPdim(\mathcal{E}_n')}=\FPdim(\mathcal{E}_n)^2.
\end{equation*}
Thus, 
$$2^{\dim(V)+1}=\FPdim(\mathcal{D})\le{\FPdim(\mathcal{E}_n})^2=2^{2(n+1)},$$ 
which implies that $\dim(V)\le 2n$.

(2) By Theorem \ref{thecenterisnondeg}, $\E_n$ is Lagrangian in $\Z(\E_n)$. 
Clearly, $\E_n$ is Lagrangian in every braided tensor subcategory $\D$ of $\Z(\E_n)$ containing it. Now it is clear that for every $n\le \ell\le 2n$, there exists a braided tensor subcategory $\D$ of $\Z(\E_n)$ of Frobenius-Perron dimension $2^{\ell+1}$ containing $\E_n$ as a Lagrangian subcategory. Since by (\ref{shimizu}), $\FPdim(\mathcal{D}){\FPdim(\mathcal{D}'\cap{\mathcal{E}_n})}=\FPdim(\mathcal{E}_n)^2$, we see that $\D$ is degenerate and $\FPdim(\D')=2^{2n-\ell+1}$, as desired.
\end{proof}

\section{Appendix: Dimension $16$}\label{appendix}

Let $E(2)$ be the $8$-dimensional Nichols Hopf algebra, with grouplike element $u$ and $(1,u)$-skew-primitive elements $x,y$ (see \ref{section:Rep(W)}). It is well known (and straightforward to check) that the dual Hopf algebra $E(2)^*$ is unimodular (i.e., $1$ is the distinguished grouplike element of $E(2)$), and that every integral of $E(2)^*$ is of the form 
$$\lambda_s:=s(xy)^*+s(uxy)^*,\,\,\,s\in\mathbb{C}.$$

Let $R_0$ be as in (\ref{r0}), and consider the $R$-matrix  
\begin{eqnarray*}\label{rmatrix}
\lefteqn{R:=R_0- 
\frac{1}{2}\left(x\otimes uy + ux\otimes uy + x\otimes y - ux\otimes y \right)}\\
& + & \frac{1}{2}\left(y\otimes ux + uy\otimes ux + y\otimes x - uy\otimes x \right)\\
& - & \left(xy\otimes xy + uxy\otimes xy + xy\otimes uxy - uxy\otimes uxy \right).
\end{eqnarray*}
Clearly $R$ is non-degenerate. It is well known \cite{G} that $(E(2),R)$ is a quasitriangular Hopf algebra with Drinfeld element ${\bf u}:=u(1+2xy)$, and that the finite braided tensor category $\Rep(E(2),R)$ is slightly-degenerate (so, $\Rep(E(2),R)\cong \D_1$ (\ref{dn}) as braided tensor categories).

Let
\begin{eqnarray*}\label{gamma1}
\lefteqn{\gamma:=R_0+ 
\frac{i}{2}\left(x\otimes uy - ux\otimes uy + x\otimes y + ux\otimes y \right)}\\
& - & \frac{i}{2}\left(y\otimes ux - uy\otimes ux + y\otimes x + uy\otimes x \right)\\
& - & \frac{1}{2}\left(xy\otimes xy + uxy\otimes xy + xy\otimes uxy - uxy\otimes uxy \right).
\end{eqnarray*}
The following can be verified in a straightforward manner.
 
\begin{lemma}\label{claim1}
The $2$ triples $(\gamma,\lambda_{\pm i},g=1)$ satisfy all the conditions in \cite[Theorem 1]{DR}. Namely, 
we have
\begin{enumerate}
\item
$\gamma$ is non-degenerate.

\item
$(\id\otimes \varepsilon)(\gamma)=1=(\varepsilon\otimes \id)(\gamma)$.

\item
$(\id\otimes S)(\gamma)=(S\otimes \id)(\gamma)$. 

\item
$(\Delta\otimes \id)(\gamma)=\gamma_{13}\gamma_{23}$ and $(\id\otimes \Delta)(\gamma)=\gamma_{12}\gamma_{13}$.

\item
$(\lambda_{\pm i}\otimes \lambda_{\pm i})((\id\otimes S)(\gamma))=1$.

\item
$(\id\otimes S^2)(\gamma)=\gamma_{21}$. \qed
\end{enumerate}
\end{lemma}
\noindent (In the language of \cite[Theorem 1]{DR}, (1)-(4) say that $\gamma$ is a non-degenerate Hopf-copairing.)

Set $a:=(1-i)/2$ and $\zeta:=e^{\pi i/4}$. Let
\begin{equation*}\label{sigma}
\sigma_+:=(a1+\bar{a}u)(1+xy),\,\,\,\sigma_-:=u\sigma_+=\sigma_+ u \in E(2).
\end{equation*}
It is easy to check that $\sigma_+,\sigma_-$ are invertible with inverses
$$\sigma_+^{-1}=(\bar{a}1+au)(1-xy),\,\,\,\sigma_-^{-1}=u\sigma_+^{-1}=\sigma_+^{-1}u,$$
and $\sigma_+^2=\sigma_-^2={\bf u}$ (hence $S^2(h)=\sigma_+^2h\sigma_+^{-2}$ for every $h\in E(2)$). 

Now using the above properties of $\sigma_+,\sigma_-$ and the properties of the $R$-matrix $R$ it is straightforward to verify the following.
 
\begin{lemma}\label{claim2}
Each one of the following $8$ sextuplets $$(R,\sigma_{\pm},g=1,\lambda_{i},\gamma,\beta=\pm\zeta)\,\,\,\text{and}\,\,\,(R,\sigma_{\pm},g=1,\lambda_{-i},\gamma,\beta=\pm i\zeta)$$ satisfies all the conditions in \cite[Theorem 3]{DR}. Namely, let $\sigma=\sigma_{\pm}$,  then we have
\begin{enumerate}
\item
$\gamma=(\sigma^{-1}\otimes 1)\Delta(\sigma)(1\otimes\sigma^{-1})$.

\item
$\lambda_{\pm i}(S(h))=\lambda_{\pm i}(\sigma h\sigma^{-1})$ for every $h\in E(2)$.

\item
$\lambda_{i}(\sigma)=\beta^2$ and $\lambda_{- i}(\sigma)=\beta^2$.

\item
The map $E(2)\to E(2)^{{\rm cop}}$, $h\mapsto \sigma h \sigma^{-1}$, is a Hopf algebra isomorphism. 

\item
$S(\sigma)=\sigma$. 

\item
$\gamma=\sum \sigma\gamma_{1}\sigma^{-1}\otimes \sigma^{-1}S(\gamma_{2})\sigma$, where $\gamma=\sum \gamma_{1}\otimes \gamma_{2}$. \qed
\end{enumerate}
\end{lemma}

In conclusion, Lemmas \ref{claim1} and \ref{claim2} establish that each one of the following $8$ sextuplets $$(R,\sigma_{\pm},g=1,\lambda_{i},\gamma,\beta=\pm\zeta)\,\,\,\text{and}\,\,\,(R,\sigma_{\pm},g=1,\lambda_{-i},\gamma,\beta=\pm i\zeta)$$ determines a structure of a finite braided tensor category on the category $\Rep(E(2))+\Vect$, in the manner prescribed by Davydov-Runkel in \cite{DR1,DR}. Clearly, these categories are {\em not integral} and have Frobenius-Perron dimension $16$, and it is straightforward to verify that they are non-degenerate and contain $\sRep(\mathbb{C})$ as a Lagrangian subcategory. 

\begin{lemma}\label{claim3}
The $8$ finite braided tensor categories constructed above are pairwise non-equivalent as braided tensor categories. 
\end{lemma}

\begin{proof}
Let $\C=(R,\sigma,g=1,\lambda,\gamma,\beta)$ be one of the $8$ finite braided tensor categories constructed above. Let $X$ be the unique non-invertible simple object of $\C$, and let $\chi$ be the unique non-trivial character of $E(2)$ (viewed as the unique non-trivial invertible object of $\C$). We have, $\chi(\sigma_{\pm})=\mp i$. 

Then it is straightforward to verify that the action of the braiding isomorphism $X\otimes _{\C}X\xrightarrow{\cong} X\otimes _{\C}X$ on the $1$-dimensional space $\Hom_{\C}({\bf 1},X\otimes _{\C}X)$ is given by multiplication by $\beta$.

It is also straightforward to verify that the braiding isomorphism $\chi\otimes _{\C}X\xrightarrow{\cong} X\otimes _{\C}\chi$ is given by $\chi(\sigma)\cdot \id_{X}$.

It thus follows from the above that the braided tensor equivalence class of $\C$ is determined by $\sigma$ and $\beta$, which implies the claim.
\end{proof}

It thus follows from the above and Theorem \ref{main theorem} that all $8$ equivalence classes of finite non-degenerate braided non-integral tensor categories $\C$, containing a Lagrangian subcategory $\sRep(\mathbb{C})$, arise from the construction of Davydov-Runkel.

\end{document}